\documentclass[english,reqno]{amsart}

\usepackage{amsmath, appendix, ulem, amsthm}
\usepackage[nobysame]{amsrefs}
\usepackage{amssymb, color}
\usepackage[margin=1in]{geometry}
\usepackage{mathrsfs}
\usepackage{graphicx}
\usepackage{subfig}
\usepackage{float}
\usepackage{epsf}
\usepackage[colorlinks=true]{hyperref}
\usepackage{titletoc}

\graphicspath{{../Figures/}}

\numberwithin{equation}{section}

\newtheorem{lem}{Lemma}[section]
\newtheorem{thm}{Main Theorem}
\theoremstyle{remark}

\renewcommand{\hat}{\widehat}

\newcommand{\f}{\frac}
\newcommand{\p}{\partial}
\newcommand{\beq}{\begin{equation}}
\newcommand{\eeq}{\end{equation}}	

\newcommand{\R}{{\mathbb R}}

\newcommand{\del}{\partial}

\newcommand{\dt}{ \, {\rm d} t}
\newcommand{\dx}{ \, {\rm d} x}
\newcommand{\dy}{ \, {\rm d} y}
\newcommand{\dz}{ \, {\rm d} z}
\newcommand{\dxi}{ \, {\rm d} \xi}
\newcommand{\dv}{ \, {\rm d} v}

\newcommand{\Eps}{\epsilon}
\newcommand{\Ni}{\noindent}

\newcommand{\BigO}{{\mathcal{O}}}

\newcommand{\CalD}{{\mathcal{D}}}

\newcommand{\RR}{\mathbb{R}}

\newcommand{\VV}{\mathbb{V}}

\newcommand{\abs}[1]{\left\lvert#1\right\rvert}

\newcommand{\vint}[1]{\left\langle#1\right\rangle}

\newcommand{\vpran}[1]{\left(#1\right)}

\begin{document}

\title{The fractional diffusion limit of a kinetic model with biochemical pathway}
\author{
Beno\^ \i t Perthame
\and
Weiran Sun
\and
Min Tang
}
\address{Sorbonne Universit\'e, UPMC Univ Paris 06,  Laboratoire Jacques-Louis Lions  UMR CNRS 7598, Inria de Paris, team Mamba, F75005 Paris, France.
\newline
This author has received funding from the European Research Council (ERC) under the European UnionÕs Horizon 2020 research and innovation programme (grant agreement No 740623). }

\address{Department of Mathematics, Simon Fraser University, 8888 University Dr., Burnaby, BC V5A 1S6, Canada.
\newline
This author is partially supported by NSERC Discovery Grant No. R611626.}

\address{Institute of natural sciences and department of mathematics, 
Shanghai Jiao Tong University, Shanghai, 200240, China. 
\newline
This author is partially supported by NSFC 11301336
and 91330203.}
\date{\today}
\maketitle

\begin{abstract}
Kinetic-transport equations that take into account the intra-cellular pathways are now considered as the correct description of bacterial chemotaxis by run and tumble. Recent mathematical studies have shown their interest and their relations to more standard models. Macroscopic equations of Keller-Segel type have been derived using parabolic scaling. 
Due to the randomness of receptor methylation or intra-cellular chemical reactions,  noise occurs in the signaling pathways and affects the tumbling rate. Then, comes the question to understand  the role  of an internal noise on the behavior of  the full  population.   
In this paper we consider a kinetic model for chemotaxis which includes biochemical pathway with noises. We show that under proper scaling and conditions on the tumbling frequency as well as the form of noise, fractional diffusion can arise in the macroscopic limits of the kinetic equation. This gives a new mathematical theory about how long jumps can be due to the internal noise of the bacteria. 
\end{abstract}
\bigskip

\noindent{\makebox[1in]\hrulefill}\newline
2010 \textit{Mathematics Subject Classification.} 35B25; 35R11; 82C40; 92C17
\newline\textit{Keywords and phrases.}  Mathematical biology; kinetic equations;  chemotaxis; asymptotic analysis; run and tumble; biochemical pathway; Fractional Laplacian ; L{\'e}vy walk.

\section*{Introduction}

Kinetic-transport equations are often used to describe the population dynamics of  bacteria moving by run-and-tumble. One of the key biological properties relating to bacteria movement is how a bacterium determines its tumbling frequency. The tumbling frequency is the rate for a running bacterium to stop and change its moving direction. Recently it has been found that, for a large class of bacteria, the tumbling frequencies depend on the level of the external chemotactic signal as well as the internal states of the bacteria. Based on this observation, kinetic models incorporating the intracellular chemo-sensory system are introduced in \cite{ErbanOthmer04,SWOT}, which write
\beq\label{eq:model-no-noise}
\del_t q
+ v  \nabla_\mathbf{x} q
+ \del_y \vpran{f(y,S)q}
= \Lambda(y,S) (\vint{q} - q) \,.
\eeq
Here $q(t,\mathbf{x},\mathbf{v},y)$ denotes the probability density function of bacteria at time $t$, position $\mathbf{x}\in\mathbb{R}^d$, velocity $\mathbf{v}\in\VV$  with $\VV$ the sphere (or the ball) with radius $V_0$, 
and the intra-cellular molecular content $y \in \R$. The function $f(y,S)$ takes into account the slowest reaction in the chemotactic signal transduction pathways for a given external effective signal $S$. The right hand side terms in \eqref{eq:model-no-noise} describes the velocity jump process where $\Lambda(y,S)$ is the tumbling frequency.
The specific forms of $f(y,S)$ and $\Lambda(y,S)$ depend on different types of bacteria, where a linear cartoon description for $f(y,S)$ is used in \cite{ErbanOthmer04} and more sophisticated forms for E.coli chemotaxis have been studied in~\cite{JOT,OXX}.  The frequency $\Lambda(y,S)$ is determined by the regulation of   the flagellar motors by biochemical pathways \cite{JOT} and it usually has steep transition with respect to $y$.

In the case when the external signal $S$ is absent, macroscopic models have been derived from~\eqref{eq:model-no-noise} in the diffusion regime. For example, in \cite{ErbanOthmer04, ErbanOthmer07, DS05, STY, X} the authors  have recovered the Keller-Segel type of equations that govern the dynamics of cell density as the diffusion limit of~\eqref{eq:model-no-noise}. These results indicate the underlying microscopic dynamics of the bacteria follow the Brownian motion.   

Recent experiments of tracking individual cell trajectories, however, showed that some bacteria actually adopts a L{\'e}vy-flight type movement instead of the Brownian motion \cite{BGM, ARBPHB}. L{\'e}vy flight is a random process whose path length distribution obeys a power-law decay, as opposed to the Brownian motion whose path length distribution decays exponentially. Therefore, a L{\'e}vy flight exhibits a non-negligible probability of "long jumps". Various explanations have been proposed to understand the origin of the long jumps. For example, the works in \cite{KEVSC,TG} relate molecular noise to power-law switching in bacterial flagellar motors. The model in \cite{Matt09} suggests that the fluctuation in CheR (a protein which regulates the receptor activity) can induce the power-law distribution of the path length.

Motivated by the aforementioned experimental and theoretical work, we study in this paper a kinetic model that incorporates noise in the intra-cellular molecular content $y$ in equation \eqref{eq:model-no-noise}. Similar equation has appeared in \cite{PTV}. Our main goal is to rigorously derive fractional diffusion equations (which correspond to L{\'e}vy processes) from the new kinetic equation. The particular equation we consider is as follows: 
\begin{align} \label{eq:scaled-kinetic}
   \Eps^{1+\mu} \del_t q_\Eps
+ \Eps v \cdot  \nabla_\mathbf{x} q_\Eps
- \Eps^s\del_y &\vpran{D(y) Q_0(y) \del_y\frac {q_\epsilon}{Q_0} }
= \Lambda(y) (\vint{q_\Eps} - q_\Eps) \,,
\\
q_\Eps (0, x, v, y) & = q^{in}(x, v, y) := \rho^0(x) Q_0(y) \geq 0 \,,
\end{align}
where $0 < \mu < 1$, $0< s < 1+ \mu$, and 
\begin{align*}
    \vint{q_\Eps}(t,x,y) := \int_\VV q_\Eps (t,x,v,y) dv \, ,
\end{align*}
with $\VV$ being the sphere $\p B(0, V_0) \subseteq \R^d$ and $dv$ is the normalized surface measure. For later purpose, we also introduce the notation
$$
 \rho_\Eps(t,x)= \int_\R \vint{q_\Eps}(t,x,y) dy.
$$

The given function $Q_0(y)$ can be viewed as the equilibrium distribution in $y$ in absence of outside signal. 
One can decompose the $y$ derivative term on the left hand side of \eqref{eq:scaled-kinetic} into two terms
$$
\Eps^s\del_y \vpran{D(y) Q_0(y) \del_y\frac {q_\epsilon}{Q_0} }=\Eps^s\del_y \vpran{D(y)  \del_y q_\epsilon }-\Eps^s\del_y \vpran{D(y)  \frac {\del_y Q_0}{Q_0}q_\Eps }.
$$
Therefore,  $D(y)$ turns out to be the diffusion coefficient in $y$.
Compared with the model in \eqref{eq:model-no-noise}, 
the diffusion term in $y$ takes into account the intrinsic noise of the signally pathway. For technical reasons we consider a specific form of noise and leave open the derivation with more general types. The initial datum $q^{in}(x,y,v)$ is assumed to be independent of $\Eps$ and takes a separated form for simplicity. One can also consider the more general case where the sequence of initial data converges as $\Eps \to 0$.

We identify conditions on the parameters and coefficients that give rise to a fractional diffusion limit as $\Eps \to 0$. We will show that under these conditions, there exists $\rho(t, x)$ such that the density function $q_\Eps$ satisfies
\begin{equation}
q_\Eps (t,x,v,y)  \to \rho(t,x) Q_0(y)
\qquad
 \text{as $\Eps \to 0$}
\label{limi_gen}
\end{equation}
and $\rho$ solves
\begin{equation} \left\{ \begin{array}{l}
 \del_t  \rho(t,x) + \nu  \vpran{-\Delta}^{\frac{1+\mu}{2}}\rho=0,
\\ 
 \rho(0,x) = 
 \rho^0(x) \,,
\end{array}\right. 
\label{eq:fractional}
\end{equation}
where the constant $\nu>0$ can be computed explicitly. 

Deriving fractional diffusion models from a classical kinetic model (where the density function only depends on $(t, x, v)$) is initiated in \cite{Olla2009} by probabilistic methods and \cite{CMT,MMM,AMP} by analytic methods. The case of boundary conditions is treated in \cite{BGM2017}. In these works, the fractional diffusion arises either from a fat-tail equilibrium distribution in the velocity $v$ \cite{CMT,MMM, AMP} or the degeneracy of the collision frequency for small velocities \cite{Olla2009, AMP}.  In some recent works in \cite{AS1,Bell2016}, similar results have been extended to kinetic models for chemotaxis, where a fractional diffusion equation with advection  is derived when there exist small bias along the direction of the chemical gradient. We note that, in all previous works for chemotaxis,  the fractional diffusion occurs from fat tail distribution with unbounded  velocities $v$, while in chemotaxis it is more realistic to consider bounded bacteria velocities. This is our main contribution, to perform a rigorous derivation with the more physical assumption of  bounded velocities.  There are also works deriving fractional diffusion limits from kinetic equations with extended variables. For example, the models in \cite{FS, EGP} have the free path length as an independent variable and fraction diffusion limits are derived under the condition that the second moments of the path length distribution functions are unbounded. The models in \cite{FS, EGP} phenomenologically incorporate occasional long jumps in the tumbling frequency, while $\Lambda(y)$ in our model depends on the internal state.

In proving the fraction diffusion limit, we note two main differences in our methodology compared with earlier works. First, unlike in the (fractional) diffusion limits of classical kinetic equations (with only $(t, x, v)$ as their independent variables), the mass conservation equation in terms of $\rho_\Eps = \int_{\R^d} \int_{\VV} q_\Eps \dv\dy$ does not seem to be the proper setting for deriving the limiting equation. This is indeed due to the appearance of the extended variable $y$ and the additional noise term. Instead, we need to consider a properly weighted quantity $\int_{\R^d}\int_{\R^d} \int_{\VV} \chi_0 q_\Eps \dv\dy\dx$ where $\chi_0$ satisfies the dual equation given by~\eqref{eq:chi-0}. This weighted quantity thus encodes the effect of the noise. 
We note that working with a weighted density seems to be a general setting when deriving (fractional) diffusion limits of kinetic equations with extended variables. See for example in \cite{FS}, where the macroscopic equations for a non-classical kinetic equation are derived for the weighted density function against the path length distribution.  Compared with~\cite{FS}, the choice of the weight function $\chi_0$ in this paper is much less obvious.
Second, the derivation of the fractional diffusion equations in~\cite{CMT,MMM,AMP} relies on the method of auxiliary functions or a related Hilbert expansion. In the current paper,  we use the method of moments \cite{BSS} which leads  to reformulate the equation for $q_\Eps$ in a convenient way (see~\eqref{eq:hatqeps}) and apply it in the flux term of the conservation law. This  framework is more standard, intuitive and consistent with the classical Chapman-Enskog method of deriving macroscopic limits of kinetic equations.  
\\


The paper is organized as follows. We begin with stating our assumptions on the parameter range and the main result, i.e., the validity of~\eqref{eq:fractional}. The proof uses the two next sections. We first state several a priori bounds and estimates which are used several times in the  main core of the proof, which is given in Section~\ref{sec:asymptotics}. 

\section{Assumptions and main results}
\label{sec:assum}

\Ni {\bf Assumptions on the coefficients. } Let $M_0 > 1$, $A_0$, $A_1$  be  fixed numbers..  We are given a smooth function $Q_0(y)$ which describes the equilibrium in the internal state $y$, 
\begin{align} \label{def:q0}
     Q_0(y) 
     =  \begin{cases}
          c^+ |y|^{-\sigma} \,, & y > M_0 \,, \\[2pt]
          c^- |y|^{-\sigma} \,, & y < -M_0 \,,
         \end{cases}
        \qquad \qquad \sigma>1,  \qquad  Q_0(y) >0, \qquad   \int_\R Q_0\,dy=1 .
\end{align}

The mechanism at work here is the degeneracy of  the tumbling rate $\Lambda$, a smooth function on $\R$, namely
\begin{align} \label{def:Lambda}
   \Lambda(y) 
   = \begin{cases} 
        \BigO(1) \,, & y \geq M_0 \,, \\[2pt]
        |y|^{-\beta} \,, & y \leq -M_0 \,,
      \end{cases} 
\qquad \quad
   \abs{\Lambda'(y)} \leq \frac{A_0}{y^{\gamma}} 
\quad
\text{for $y > M_0$} \,,
\end{align}

Assume that the diffusion coefficient $D$ is a smooth functions on $\R$ such that
\begin{align} \label{def:D}
   D(y) 
   = \begin{cases}
       \BigO(1) \,, & y \in [-M_0, M_0] \,, \\[2pt]
       A_1 |y|^{n+1} \,, & |y| \geq M_0 \,.
      \end{cases}
\end{align}
for some $n > 0$ whose range will be specified in~\eqref{assump:main}. The conditions on $\sigma, \beta, \gamma$ are also described in~\eqref{assump:main}.

\medskip

\Ni {\bf Assumptions on the initial data.} We assume that, for some constant $B$, 
\begin{align}\label{ID}
q_0 \leq B Q_0, \qquad \quad     \int_\R \int_\R \int_\VV
     \frac{q_0^2}{Q_0} ( x, v, y)dv dy dx  \leq B, \qquad \quad  \int_\R \int_\R \int_\VV
q_0 ( x, v, y)dv dy dx \leq B \,.
\end{align}

\medskip

\Ni {\bf Parameter range.} The main assumptions of the parameters are 
\begin{align} \label{assump:main}
     n > \sigma > 1 \,, 
\qquad s > 1 \,,
\qquad \gamma>\frac{n-\sigma}{2} + 1\,,
\qquad \beta > n - 1\,,
\qquad
     \beta + n - 1 > s \beta > \beta + \sigma - 1\,.
\end{align}

The analysis below  leads to the relation  
\begin{equation} \label{value_mu}
\mu = \frac{n-1}{\beta} \in (0, 1) \, ,
\end{equation}
therefore, we observe that 
\begin{align*}
    \beta + n - 1 > s \beta
\Longleftrightarrow
    1 + \mu > s \,,
\end{align*}
which makes the time-derivative term in equation~\eqref{eq:scaled-kinetic} a (formally) high-order term. 
\\

Then, we have the 

\begin{thm}
Let $q_\Eps$ be the solution of~\eqref{eq:scaled-kinetic} with the above assumptions \eqref{def:q0}--\eqref{ID}. Suppose the parameters $n, \sigma, s, \beta, \gamma$ satisfy the parameter range~\eqref{assump:main}. Then, as $\Eps \to 0$, the limit \eqref{limi_gen} holds in the sense that $\frac{q_\Eps}{Q_0}$ converges  $L^\infty-w*$ to  $\rho \in L^\infty(\R^+; L^1\cap L^\infty(\R^d))$ and $\rho$ satisfies the 
fractional Laplacian equation~\eqref{eq:fractional}. 
\label{thm:main}
\end{thm}

The end of the paper is devoted to the proof.

\section{Estimates and a priori bounds} 

\subsection{Relative entropy estimates}
The method of relative entropy can be applied to provide us with useful a priori bounds for all $t \geq 0$:
\begin{align}\label{L2-infty}
 0\leq q_\Eps \leq B Q_0, \qquad  \int_\R \int_\R \int_\VV
     \frac{q_\Eps^2}{Q_0} (t, x, v, y) \dv \dy \dx \leq B,  \qquad  \int_\R \int_\R \int_\VV
q_\Eps (t, x, v, y)\dv \dy \dx \leq B \, ,
\end{align}
and
\begin{align} \label{bound:orthogonal}
  \int_0^\infty \int_\R \int_\R \int_\VV
     D(y) Q_0(y) \vpran{\del_y \vpran{\frac{q_\Eps}{Q_0}}}^2
\leq B \Eps^{1+\mu-s} \,,
\qquad
   \int_0^\infty \int_\R \int_\R \int_\VV 
      \Lambda(y) \frac{\vpran{q_\Eps - \vint{q_\Eps}}^2}{Q_0}
\leq 
  B  \Eps^{1+\mu}  \,.
\end{align}
The derivation of these estimates follows from multiplying equation~\eqref{eq:scaled-kinetic} by $\displaystyle \frac{q_\Eps}{Q_0}$ and integrating in $x, v, y$. The resulting equation is
\begin{align*}
   \frac{1}{2}\Eps^{1+\mu} \frac{\rm d}{\dt}
   \int_\R \int_\R \int_\VV
     \frac{q_\Eps^2}{Q_0}
  +  \Eps ^s \int_\R \int_\R \int_\VV
          D(y) Q_0 \vpran{\del_y \vpran{\frac{q_\Eps}{Q_0}}}^2
   + \int_\R \int_\R \int_\VV 
      \Lambda(y) \frac{\vpran{q_\Eps - \vint{q_\Eps}}^2}{Q_0} = 0 \,.
\end{align*}

A first and immediate consequence of these estimates is the weak convergence of $q_\Eps$
\begin{lem} \label{lem:wcv} 
After extraction of a subsequence, still denoted by $q_\Eps$, we have
$$
 \f{q_\Eps }{Q_0} (t, x, v, y) \to \rho(t,x) , \qquad \text{in } L^\infty(\R^+ \times \R^d \times \R \times \VV)-w^\ast \, ,
$$
where $\rho(t,x)\in   L^\infty (\R^+ ; L^1\cap L^\infty (\R^d))$. 
\end{lem}

\subsection{A priori bounds}

Another   consequence of the a priori estimate is the following lemma:
\begin{lem} \label{lem:a-priori}
Suppose $q_\Eps$ satisfies  the a priori bound~\eqref{bound:orthogonal}. Denote 
\begin{align*}
   R_\Eps = \int_\R  q_\Eps \dy \,.
\end{align*} 
Then there exists a constant $C > 0$ independent of $t, x, y$ and $\Eps$ such that for all $y \in \R$, we have
\begin{align} \label{bound:deviation}
 \abs{ \f{q_\Eps }{Q_0} (t, x, v, y) - R_\Eps (t, x,v)}
\leq C H^{1/2}(t, x,v)  \,, \qquad \forall y \in \R,\, v\in\VV, 
\end{align}
where 
\begin{align} \label{def:H}
H(t, x,v) = \int_\R Q_0(y)D(y)\vpran{\p_{y'}\Big(\f{q_\Eps(t, x, y,v)}{Q_0(y)}\Big)}^2\dy  \,.
\end{align}
\end{lem}
\begin{proof}
By the a priori bound~\eqref{bound:orthogonal}, it holds that
\begin{align*}
 \abs{ \f{q_\Eps}{Q_0} - \rho_\Eps}
&= \abs{\f{q_\Eps(y)}{Q_0(y)}- \int\f{q_\Eps(z)}{Q_0(z)} Q_0(z)\dz}
\leq 
    \int_\R \abs{\f{q_\Eps(y)}{Q_0(y)}-\f{q_\Eps(z)}{Q_0(z)} }Q_0(z)\dz
\\
&= \int_\R \vpran{\int_z^y \abs{\p_{y'} \vpran{\f{q_\Eps(y')}{Q_0(y')}}}\dy'}Q_0(z)\dz
\\
 &\leq 
       \int_\R \vpran{\abs{\int_z^y Q_0(y')D(y')\vpran{\p_{y'}
                              \vpran{\f{q_\Eps(y')}{Q_0(y')}}}^2\dy'}}^{1/2}
                    \vpran{\abs{\int_z^y\f{1}{Q_0(y')D(y')}\dy'}}^{1/2} Q_0(z)\dz
\\
&\leq 
    \vpran{\abs{\int_\R \frac{1}{Q_0(y')D(y')}\dy'}}^{1/2}   H^{1/2}(t, x,v) \,.
\end{align*}
Near $y = \pm \infty$, we have
\begin{align*}
    Q_0(y) \sim |y|^{-\sigma} \,,
\qquad
    D(y) \sim  |y|^{n+1} , \qquad 
    \frac{1}{Q_0(y')D(y')} \sim \frac{1}{|y|^{n+1-\sigma}} \,,
\end{align*}
which is integrable on $\R$ by the assumption that $n > \sigma$. Hence~\eqref{bound:deviation} holds with the constant $C = \vpran{\int_\R \frac{1}{Q_0(y')D(y')} \dy'}^{1/2}$.
\end{proof}

\subsection{From the Fourier side}

In fact, we need Fourier versions of the a priori bounds and thus we denote the Fourier transform in $x$ of $u$ with a $\hat u$, in particular 
$$
 \hat q (t, \xi, v, y) =\int_{\R^d} q(t,x,v,y)e^{ix.\xi} dx.
$$
For instance, from \eqref{bound:orthogonal}, we conclude, using Parseval identity,
\begin{align} \label{bound:fourierL2v}
 \int_0^\infty \int_{\R^d} \int_\R \int_\VV 
      \Lambda(y) \frac{\vpran{\hat q_\Eps - \vint{\hat q_\Eps}}^2}{Q_0}
\leq 
  B  \Eps^{1+\mu}  \,.
\end{align}

Also, following the same calculations as in Lemma~\ref{lem:a-priori}, we have
\begin{align} \label{bound:fouriery}
\abs{ \f{\hat q_\Eps }{Q_0} (t, \xi, v, y) - \hat R_\Eps (t, \xi,v)}
\leq C K^{1/2}(t, \xi,v)  \,, \qquad \forall y \in \R,\, v\in\VV, 
\end{align}
with
\begin{align} \label{def:K}
K(t, \xi,v) = \int_\R Q_0(y)D(y)\abs{\p_{y}\Big(\f{\hat q_\Eps(t, \xi, y,v)}{Q_0(y)}\Big)}^2 \dy  \,.
\end{align}
And Parseval identity gives  
\begin{align} \label{est:Kfourier}
\int_0^\infty \int_\VV \int_{\R^d}  K(t, \xi ,v) \dxi \dv   \dt = \int_0^\infty \int_\VV \int_{\R^d}  H(t, x,v) \dx \dv   \dt  \leq  B  \Eps^{1+\mu-s}  \, .
\end{align}

Because, for any $M_1>0$
\begin{align*}
\f 12 \int_0^\infty \int_{\R^d} \int_{y>-M_1}  \int_\VV  \frac{\vpran{\hat q_\Eps - \hat \rho_\Eps Q_0}^2}{Q_0} \leq&
\int_0^\infty \int_{\R^d} \int_{y>-M_1}  \int_\VV Q_0   \vpran{\frac{ \hat q_\Eps} {Q_0} - \frac{\vint{\hat q_\Eps}} {Q_0} }^2
\\ &+\int_0^\infty \int_{\R^d} \int_{y>-M_1}  Q_0  \vpran{ \frac{\vint{\hat q_\Eps}} {Q_0} - \vint{\hat R_\Eps}}^2 .
\end{align*}
Finally, combining \eqref{bound:fourierL2v}, \eqref{bound:fouriery} and \eqref{est:Kfourier}, we also infer that, in Fourier variable, we have  for all $M_1>0$, 
\begin{align} \label{bound:fourierL2}
 \int_0^\infty \int_{\R^d} \int_{y>-M_1}  \int_\VV  \frac{\vpran{\hat q_\Eps - \hat \rho_\Eps Q_0}^2}{Q_0} \leq C \Eps^{1+\mu-s} \, .
 \end{align}
 
\subsection{Useful calculations}
Two integrals repeatedly appear in the rest of this note. We list them out as a lemma:
\begin{lem} \label{lem:integral}
Suppose 
\begin{align*}
   0 < \alpha + 1 < 2\beta_1 \,,
\qquad
  0 < \alpha + 1 <  \beta_2 \,,
\qquad
  \beta_1, \beta_2 > 0 \,.
\end{align*}
Then the following integrals are well-defined and there exists a constant $c_1 > 0$ such that
\begin{align*}
   \int_{-\infty}^0 \frac{|y|^\alpha}{1 + \vpran{\Eps |\xi \cdot v|  |y|^{\beta_1}}^2} \dy
= c_1 \vpran{\Eps |\xi \cdot v| }^{-\frac{\alpha+1}{\beta_1}} \,,\qquad
\int_{-\infty}^0 \frac{|y|^\alpha}{\sqrt{1 + \vpran{\Eps |\xi\cdot v|  |y|^{\beta_2}}^2}} \dy
= c_2 \vpran{\Eps  |\xi\cdot v| }^{-\frac{\alpha+1}{\beta_2}} \,.
\end{align*}
\end{lem}
\begin{proof}
Make a change of variable $z = \Eps |\xi \cdot v|  |y|^{\beta_1}$ in the first integral and $z = \Eps |\xi \cdot v|  |y|^{\beta_2}$ in the second one. Then
\begin{align*}
   \int_{-\infty}^0 \frac{|y|^\alpha}{1 + \vpran{\Eps  |\xi\cdot v|  |y|^{\beta_1}}^2} \dy
= \frac{1}{\beta_1} \vpran{\Eps |\xi\cdot v| }^{-\frac{\alpha+1}{\beta_1}}
    \int_0^\infty \frac{z^{\frac{\alpha + 1}{\beta_2} - 1}}{1 + z^2} \dz
= c_1 \vpran{\Eps  |\xi\cdot v| }^{-\frac{\alpha+1}{\beta_1}} \,,
\end{align*}
\begin{align*}
   \int_{-\infty}^0 \frac{|y|^\alpha}{\sqrt{1 + \vpran{\Eps  |\xi\cdot v|  |y|^{\beta_2}}^2}} \dy
= \frac{1}{\beta_2} \vpran{\Eps |\xi\cdot v| }^{-\frac{\alpha+1}{\beta_2}}
    \int_0^\infty \frac{z^{\frac{\alpha + 1}{\beta} - 1}}{\sqrt{1 + z^2}} \dz
= c_2 \vpran{\Eps |\xi\cdot v| }^{-\frac{\alpha+1}{\beta_2}} \,,
\end{align*}
where the integrability of the $z$-integral is guaranteed respectively by the condition $0 < \frac{\alpha+1}{\beta_1} < 2$ and $0 < \frac{\alpha+1}{\beta_2} < 1$  , or equivalently, $0 < \alpha + 1 < 2 \beta_1$ and $0 < \alpha + 1 <  \beta_2$.
\end{proof}

\section{Asymptotics} 
\label{sec:asymptotics}

\subsection{A solution of the dual problem}

We are going to make use of a weight in the variable $y$  that is built by duality.  
Let  $\chi_0(y)$ be given by 
\begin{align} \label{def:chi-0}
   \chi_0(y) 
   = \int_{-\infty}^y \frac{1}{D(z)Q_0(z) } \dz \,.
\end{align}
It is a solution of the dual problem in $y$ because
\begin{align} \label{eq:chi-0}
\del_y(D(y) Q_0(y) \del_y\chi_0)=0.
\end{align}

The properties of $\chi_0$ are summarized in the following lemma:
\begin{lem} \label{lem:chi-0}
With $Q, \, D$ as in~\eqref{def:q0}, \eqref{def:D} and with the parameter range~\eqref{assump:main}, $\chi_0 \in C_b(\R)$ is nonnegative, increasing and 
\begin{align*}
   \chi_0 
   = \begin{cases}
          \BigO(1) \,, & y > -M_0 \,, \\[2pt]
        C^-  |y|^{\sigma-n} \,, & y < - M_0 \,.
       \end{cases}
\end{align*}
\end{lem}
\begin{proof}
The non-negativity and monotonicity are both clear by the positivity of $D$ and $Q_0$. We check the behaviour of $\chi_0$ near $y = \pm \infty$. Recall that $\sigma < n$. Thus for $y < - M_0$,  
\begin{align*}
   \int_{-\infty}^y \frac{1}{D(z)Q_0(z)  }  \dz = \f{1}{c^- A_1}  \int_{-\infty}^y \frac{\dz}{z^{n+1-\sigma}  }
= C^-  |y|^{\sigma-n} \,.
\end{align*}
For $y > M_0$, the same decay holds for $D$ and $Q_0$, and thus $\frac{1}{D(z)Q_0(z)  } $ is integrable and it proves that $\chi_0$ is bounded. 
\end{proof}

\subsection{The proof of Theorem \ref{thm:main}}

We derive the limiting equation by multiplying both sides of \eqref{eq:scaled-kinetic} by the  weight function $\chi_0(y)$ and integrate in $y$ and $v$. Thanks to the property that $\chi_0$ solves the dual problem in~$y$, we find 
\begin{equation}\label{conservationLaw}
\partial_t \int_\R \int_\VV q_\Eps \chi_0 \dy \dv +{\rm div}_x J_\Eps =0, \qquad \qquad J_\Eps:=\f{1}{\Eps^{\mu}}  \int_\R \int_\VV v q_\Eps \chi_0 \dy \dv .
\end{equation}
We observe that, using Lemma~\ref{lem:wcv}, the weak limit of the first term is 
$$
 \int_\R \int_\VV q_\Eps \chi_0 \dy \dv \to \int_\R \int_\VV \rho(t,x) Q_0(y) \chi_0 \dy \dv = B_0 \rho(t,x), \qquad B_0  = \int_\R Q_0(y) \chi_0 \dy \dv .
$$

It remains to identify the limit of the flux $J_\Eps$. Notice that the a priori estimates do not provide any $L^p$ bound on $J_\Eps$ and it turns out that this term is a fractional derivative in $x$. This motivates to work in the Fourier variable. 

We are going to prove that, for some constant $\nu_0$, as $\Eps \to 0$, 
\begin{equation}
   \hat{{\rm div}_ x J_\Eps} \to  \nu_0 |\xi |^{\f{n-1}{\beta} + 1}   \hat \rho \,,  
\qquad
\text{in the sense of distributions (or in $\CalD'(\R^+ \times \R^d)$)}
\label{est:RestL2}
\end{equation}
and thus conclude the proof of Theorem \ref{thm:main}. 
\subsection{Identifying the flux $J_\Eps$}

 We apply Fourier transform in $x$ for \eqref{eq:scaled-kinetic}, and denote by $\xi$  the Fourier variable. We obtain
 \begin{align*}
   \Eps^{1+\mu} \del_t \hat q_\Eps
+ i\Eps\xi \cdot v \,  \hat q_\Eps
 -\Eps^s\del_y\vpran{D(y)Q_0(y)\del_y \f{\hat q_\epsilon}{Q_0(y)}}
= \Lambda(y) (\vint{\hat q_\Eps} - \hat q_\Eps) \, ,
\end{align*}
from which,  combining the terms including $\hat q_\epsilon$, we get
\begin{align}\label{eq:hatqeps}
   \hat q_\Eps - \vint{\hat q_\Eps}
=  - \frac{i \Eps \xi \cdot v}{i\epsilon\xi \cdot v+\Lambda} \vint{\hat q_\Eps} 
   + &\frac{\Eps^s}{i\epsilon\xi \cdot v+\Lambda}\del_y\vpran{D(y)Q_0(y)\del_y \f{\hat q_\epsilon}{Q_0(y)}}
    -\Eps^{1+\mu}\frac{1}{i\epsilon\xi \cdot v+\Lambda} \del_t \hat q_\Eps  \,.
\end{align}

Therefore, we may also decompose $\hat{{\rm div}_x J_\Eps}=\f{1}{\Eps^{\mu}}\int_\R \int_\VV \vpran{ i \xi \cdot v}  \chi_0 \hat q_\Eps \dy \dv$ according to the three terms on the right hand side as 
\begin{align}
  \hat{{\rm div}_x J_\Eps}(t,\xi) 
  = \f{1}{\Eps^{\mu}}\int_\R \int_\VV \vpran{ i \xi \cdot v}  \chi_0 \vpran{\hat q_\Eps-  \vint{\hat q_\Eps}} \dy \dv
  =  i\xi\cdot\hat J_\Eps^1 +   \hat J_\Eps^2 + \partial_t    \hat J_\Eps^3.
\label{eq:Jdecomp}
\end{align}
We show in the following subsections that  the last two contributions vanish as $\Eps \to 0$ and the fractional Laplacian stems from the first term. Using the symmetry of $\VV$, the imaginary part below vanishes and we have
$$
 \hat J_\Eps^1 (t,\xi) = \f{-1}{\Eps^{\mu}} \int \int  v \chi_0 \frac{i \Eps \xi \cdot v}{i\epsilon\xi \cdot v+\Lambda} \vint{\hat q_\Eps} \dy \dv  = \f{-i}{\Eps^{\mu}} \int \int  v \chi_0 \frac{\Lambda \Eps \xi \cdot v}{(\epsilon\xi \cdot v)^2+\Lambda^2} \vint{\hat q_\Eps} \dy \dv.
$$
Therefore we may write (notice that $\hat \rho_\Eps$ is bounded in $L^2$)
$$
 \hat J_\Eps^1 = \hat \rho_\Eps \f{-i}{\Eps^{\mu}} \int \int  v \chi_0 \frac{\Lambda \Eps \xi \cdot v}{(\epsilon \xi \cdot v)^2+\Lambda^2} Q_0(y) \dy \dv +  \hat {RJ_\Eps^1} 
$$
and, because $\mu <1$, the contribution in the integral comes from the values $y \to - \infty$ where $\Lambda(y)$ vanishes. We prove next that $\hat{RJ_\Eps^1} $ vanishes.  Thus, noting that $\hat\rho_\Eps$ converges to $\hat\rho$ weakly in $L^2$, we obtain 
$$
  \hat  J_\Eps(t,\xi) \to - \hat \rho \;   \lim_{\Eps \to 0} \f{i}{\Eps^{\mu}}  \int \int v \chi_0 \frac{\Lambda \Eps \xi \cdot v}{(\epsilon\xi \cdot v)^2+\Lambda^2} Q_0(y) \dy \dv.
$$
Using  Lemma \ref{lem:integral} and with $v_1= v \cdot \xi/|\xi|$, the above limit yields  the limit of $\hat{{\rm div}_ x J_\Eps}$ such that
\begin{equation}
\hat{\rho}\f{1}{\Eps^{\mu}}   \int_\VV  \int_{-\infty}^0 \xi\cdot v \chi_0 \frac{|y|^{\beta -\sigma}\Eps \xi \cdot v}{(\epsilon\xi \cdot v |y|^{\beta})^2+1} = \hat\rho\f{1}{\Eps^{\mu}}    \int_\VV   c_1v_1|\xi|(  |v_1| \Eps |\xi |)^\f{n-1}{\beta}.  
\label{calcul_frac}
\end{equation}
This calculation gives the announced scale $\mu = \f{n-1}{\beta}$ and the fractional derivative in \eqref{eq:fractional}.

It remains to show that the other terms vanish.

\subsection{The term $\hat {RJ_\Eps^1}$}
\label{sec:rj1}

This term is
$$
\hat {RJ_\Eps^1}  = \f{-i}{\Eps^{\mu}}  \int _\VV  \int_\RR v \chi_0  \frac{\Lambda \Eps \xi \cdot v}{(\epsilon \xi \cdot v)^2+\Lambda^2} Q_0(y) \vpran{\f{ \vint{ \hat q_\Eps }}{ Q_0(y)}-  \hat  \rho_\Eps} \dy \dv .
$$
For $y > - M_0$, because $\Lambda(y)$ is bounded from below, we may use the $L^2$ bound \eqref{bound:fourierL2} and $\mu <1$ to conclude that the corresponding part vanishes. Therefore we may again consider only the tail $y<-M_0$. We control the corresponding term using estimates similar to  \eqref{calcul_frac},  by
\begin{align*}
& \f{1}{\Eps^{\mu}} \vpran{\int_{-\infty}^0 \int_\VV |v|  \chi_0  \frac{|y|^\beta \Eps |\xi \cdot v|}{(\epsilon \xi \cdot v |y|^\beta)^2+1} Q_0(y)\dy \dv}
\; \vpran{\sup_y \abs{\f{ \vint{ \hat q_\Eps (t,\xi, y) }}{ Q_0(y)}-  \hat  \rho_\Eps (t,\xi)}}
 \\
&
= C\int_\VV |v| |\xi \cdot v |^\f{n-1}{\beta}  \dv \; \sup_y \abs{ \int_\VV \f{\hat q_\Eps(t,\xi, y ,v) }{ Q_0(y)} \dv-  \int_\VV \hat  R_\Eps (t,\xi,v) \dv}
 \\
& \leq C |\xi |^\f{n-1}{\beta}  \int_\VV  \sup_y \abs{ \f{\hat q_\Eps }{ Q_0(y) }-  \hat  R_\Eps } \dv
 \\
& \leq C |\xi |^\f{n-1}{\beta} \int_\VV K^{1/2}(t,\xi, v) \dv
\leq C |\xi |^\f{n-1}{\beta} \vpran{\int_\VV K(t,\xi, v) \dv}^{1/2} \,.
\end{align*}
and we conclude,  using \eqref{est:Kfourier} because we assume $1+\mu >s$ in \eqref{assump:main}-\eqref{value_mu}, that $i\xi\cdot\hat {RJ_\Eps^1}$ vanishes in $\CalD'(\R^+ \times \R^d)$.

\subsection{The term $\hat J_\Eps^2$}
\label{sec:J2}

Back to \eqref{eq:Jdecomp}, we show that $\hat J_\Eps^2$ vanishes as $\Eps \to 0$. The term $\hat J_\Eps^2$ is given by 
\begin{align*}
\hat J_\Eps^2 & =  \Eps^{s-\mu} \int_\VV\int_\R\frac{\vpran{i \xi \cdot v}\chi_0}{i\epsilon\xi \cdot v+\Lambda}
     \del_y\vpran{D(y)Q_0(y)\del_y \f{\hat q_\epsilon}{Q_0(y)}}\dy\dv
\\
     &=   -\Eps^{s-\mu} \int_\VV\int_\R \left[\frac{\vpran{i \xi \cdot v} \del_y \chi_0}{i\epsilon\xi \cdot v+\Lambda} - \frac{\vpran{i \xi \cdot v}\chi_0  \del_y \Lambda }{(i\epsilon\xi \cdot v+\Lambda)^2} \right]
    D(y)Q_0(y)\del_y \f{\hat q_\epsilon}{Q_0(y)}\dy\dv
\end{align*}
after integrating by parts. 

Recalling the definition of $K$ in~\eqref{def:K}, and using the Cauchy-Schwarz inequality, we can get the upper bound 
\begin{align*}
 |  {\hat J_\Eps^2} |^2
&\leq C \Eps^{2(s-\mu) }  \int_\R \int_\VV  D(y) Q_0(y)
 \left[ \frac{ |\xi \cdot v|^2 (\del_y \chi_0)^2}{| \epsilon\xi \cdot v|^2 +\Lambda^2} 
 + \frac{|\xi \cdot v|^2 \chi_0^2  (\del_y \Lambda)^2 }{((\epsilon\xi \cdot v)^2+\Lambda^2 )^2 } \right] \dv\dy   \; \int_{\VV} K(t,\xi,v) \dv  
 \\ &=C \Eps^{2(s-\mu) }  \left[ G^1(t,\xi)+G^2(t,\xi) \right] \;  \int_{\VV} K(t,\xi,v) \dv .
\end{align*}

We begin with the term $G^1$. Using the definitions of $\chi_0$ in \eqref{def:chi-0}, we have
$$
G^1(t,\xi) = \int_\R \int_\VV \frac{1}{D(y) Q_0(y)} \frac{ |\xi \cdot v|^2 }{ | \epsilon\xi \cdot v|^2+ \Lambda^{2}}\dv\dy .
$$
Because, for $| y| \gg 1$, $\frac{1}{D(y) Q_0(y)} \approx | y |^{-n-1+\sigma}$ is integrable, the values $y>-M_0$ contribute to a small term and the difficulty is for $y<-M_0$. The corresponding contribution to $G^1$ is, using Lemma~\ref{lem:integral},
$$
\int_\R \int_\VV  | y |^{-n-1+\sigma}  \frac{ | y|^{2\beta} |\xi \cdot v|^2 }{1+ | \epsilon \xi \cdot v|^2 |y|^{2\beta}}\dv\dy = c \int_\VV  | \epsilon \xi \cdot v|^{\frac{n-\sigma -2 \beta}{\beta}} |\xi \cdot v|^2  \dv.
$$
Integrability in $v$ is immediate since $n> \sigma$. 
 The resulting power in $\Eps$ in the corresponding expression of $|  {\hat J_\Eps^2} |^2$ is, taking into account~\eqref{est:Kfourier}, 
$$
2(s-\mu)\;  +\frac{n-\sigma -2 \beta}{\beta} + 1\; +\mu-s = s+\frac{1-\sigma}{\beta} -1 >0
$$
thanks to the last condition in the parameter range~\eqref{assump:main}. Therefore this contribution vanishes in $L^2(\R^d)$.
\\

The term with $G^2$ is treated with different arguments depending on the values of $y$ and, because the middle range is easy  we treat separately $y> M_0$ and $y< -M_0$. For $y> M_0$, we use the condition for $\Lambda'$ in \eqref{def:Lambda} and obtain the bound by
\begin{align*}
    C \int_{y> M_0} \int_\VV  D(y) Q_0(y)  (\del_y \Lambda)^2 \ \dv\dy  
\leq
  C  \int_{y> M_0} \int_\VV |y |^{n+1-\sigma}  |y |^{-2\gamma} \ \dv\dy  
\end{align*}
which itself is bounded thanks to the parameter range $2\gamma >n+2-\sigma $ in~\eqref{assump:main}. Therefore this contribution to $G^2$ obviously vanishes. 

Finally, the contribution to $G^2$ for $y< -M_0$ is more elaborate. We have 
\begin{align*}
 \int_{y<- M_0} \int_\VV  D(y) Q_0(y) \chi_0^2  \frac{ (\del_y \Lambda)^2 |\xi \cdot v|^2 }{((\epsilon\xi \cdot v)^2+\Lambda^2 )^2 } \dv\dy & \leq C
  \int_{y<- M_0} \int_\VV  | y |^{1-n+\sigma}  \frac{| y |^{-2(1+\beta)} | y |^{4\beta} |\xi \cdot v|^2 }{(1+(\epsilon\xi \cdot v \,| y |^\beta)^2 )^2 } \dv\dy
\\
& \leq 
 C \int_\VV (\epsilon\xi \cdot v)^{\frac{n-\sigma- 2 \beta }{\beta}} |\xi \cdot v|^2 \dv 
  = C \epsilon^{\frac{n-\sigma- 2 \beta }{\beta}} |\xi|^{\frac{n - \sigma}{\beta}} .
 \end{align*}
Therefore, in $G^2$,  the  power of $\Eps$ stemming from this is 
$$
 2(s-\mu)\;  +\frac{n-\sigma- 2 \beta }{\beta} \; + 1+ \mu-s = s+ \frac{1-\sigma}{\beta} -1 >0
$$
using again the assumption~\eqref{assump:main}.

\subsection{The term $\hat J_\Eps^3$}

This term is 
\begin{align*}
\hat J_\Eps^3 (t, \xi)
= -  \Eps \int_\VV\int_\R\frac{\vpran{i \xi \cdot v}\chi_0}{i\epsilon\xi \cdot v+\Lambda} \hat q_\Eps \dy\dv \,,
\end{align*}
and we show that, for all $T>0$, this term vanishes strongly in $L^2((0,T) \times   \R^d )$ as $\Eps \to 0$. To this end, we separate the integral as
\begin{align*}
- \hat J_\Eps^3 (t, \xi)
= \Eps \int_{\VV} \int_{y>-M_0}\frac{\vpran{i \xi \cdot v}\chi_0}{i\Eps\xi \cdot v+\Lambda} \hat q_\Eps\dy\dv
+ \Eps \int_\VV \int_{y<-M_0}\frac{\vpran{i \xi \cdot v} \chi_0}{i\Eps\xi \cdot v+\Lambda} \hat q_\Eps\dy\dv \,.
\end{align*}

The term with the integration over $y > -M_0$ is easy to estimate because we control it, using the Cauchy-Schwarz inequality,  by
\begin{align*}
C \Eps \int_\VV \int_\R Q_0^{1/2}  \frac{ | \hat q_\Eps |}{Q_0^{1/2}} \dy\dv \leq   \Eps \left( \int_\VV \int_\R \frac{ | \hat q_\Eps |^2}{Q_0} \dy\dv\right)^{1/2}
\end{align*}
and this term is of order $\Eps$ in $L^2(\R^d)$ uniformly in time thanks to the second bound in~\eqref{L2-infty} which holds in Fourier variable as well. 
\\

The term with the integral over $y < -M_0$ has to be treated more carefully. Using using the Cauchy-Schwarz inequality, we have\begin{align*}
 \abs{ \Eps \int_\VV \int_{y<-M_0}\frac{\vpran{i \xi \cdot v} \chi_0}{i\Eps\xi \cdot v+\Lambda} \hat q_\Eps\dy\dv}^2
\leq 
 \Eps^2  \int_\VV \int_{y<-M_0}\frac{|\xi \cdot v|^2 \chi_0^2}{(\Eps\xi \cdot v)^2+\Lambda^2} Q_0\dy\dv \;  \int_\VV \int_{\RR}\frac{| \hat q_\Eps|^2}{Q_0} \dy\dv .
 \end{align*}
Using the assumptions in section~\ref{sec:assum} and Lemma~\ref{lem:integral}, this is also upper bounded by 
\begin{align*}
& \quad \,
  C \Eps^2  \int_\VV \int_{y<-M_0}  \frac{|\xi \cdot v|^2 |y|^{\sigma-2n+2 \beta}}{1 + (\Eps |\xi \cdot v |\, |y|^\beta)^2 } \dy dv  \;  \int_\VV \int_{\RR}\frac{| \hat q_\Eps|^2}{Q_0} \dy\dv  
\\
&\leq   
  C \Eps^2  \vpran{\int_\VV (\Eps |\xi \cdot v|)^{-\frac{\sigma-2n+2 \beta+1}{\beta}} |\xi \cdot v|^2  \dv}
  \;  \int_\VV \int_{\RR}\frac{| \hat q_\Eps|^2}{Q_0} \dy\dv  
\\
&\leq   
  C (\Eps|\xi|)^{2-\frac{\sigma-2n+2 \beta+1}{\beta}} 
  \;  \int_\VV \int_{\RR}\frac{| \hat q_\Eps|^2}{Q_0} \dy\dv.  
\end{align*}
Here integrability in $y$ and $v$ are due to the assumption that $n > \sigma >1$ in~\eqref{assump:main}. Therefore,  by the same $L^2$ bound for $\hat q_\Eps$ as above for the ``easy part'', we conclude that $\hat J_\Eps^3$ vanishes in $\CalD'(\R^+ \times \R^d)$ as $\Eps \to 0$.

\section{Conclusion}
In this work we give a new rigorous derivation of fractional diffusion limit for a bacterial population, with the remarkable feature that the speed of cells during their jump is bounded and their jumps are controlled by an internal process.
The intracellular noise can replace the infinite speed assumption in \cite{AS1, AS2}, and thus  plays an important role on the population-level behaviour for {\it E. coli} chemotaxis. 
In particular, when the intracellular noise is strong ($n>1$) and the adaptation process is slow ($s>1$), the bacteria move with a L{\'e}vy walk and their population-level behaviour turns out to satisfy a fractional diffusion equation. This is in contrast to the case when there is no noise involved and the population-level equation is a regular diffusion \cite{ErbanOthmer04,STY, X}.

Our derivation is obtained rigorously under the assumption that the parameters and coefficients satisfy \eqref{def:q0}-\eqref{assump:main}. The conditions of the coefficients in \eqref{def:q0}-\eqref{def:D} require that both the equilibrium and tumbling frequency decay polynomially with respect to the internal variable $y$ as $y\to-\infty$. Part of the assumptions for the parameters in~\eqref{assump:main} are for mathematical convenience and it is not yet clear to us whether they are biologically relevant. However, among them, the two major conditions $s>1$ and $n>1$ are consistent with those required in biophysics works \cite{TG,Matt09}, where with added noise in the chemotactic signally pathways, the authors perform stochastic simulations and obtain path length distributions with polynomial tails that correspond to L\'{e}vy processes.

Several points remain to undersstand. The case where the structuring variable is time between jumps, proposed in \cite{EGP}  is a possible direction.  Also, other scalings in the model with internal pathwayl are certainly possible. Finally, our current work does not contain chemical signals.  In the presence of this exterior influence, the bacteria move towards their favorite location by advection or advection/diffusion, see \cite{ST}. One interesting question is
how intracellular noise can affect the advection with the appearance of chemical signals. This will be for our future investigation.

\bibliographystyle{amsxport}
\bibliography{Ref}

\end{document}